\documentclass[a4paper,12pt,oneside,reqno]{amsart}

\usepackage{amsfonts}
\usepackage{amsmath}   
\usepackage{amssymb}
\usepackage{amsthm}
\usepackage{eucal}

\newenvironment{spmatrix}
{\footnotesize\begin{pmatrix}}
	{\end{pmatrix}}

\usepackage{hyperref}

\allowdisplaybreaks[4]

\newcommand{\half}{{\frac{1}{2}}}

\theoremstyle{plain}
\newtheorem{theorem}{Theorem}
\newtheorem{proposition}[theorem]{Proposition}
\newtheorem{lemma}[theorem]{Lemma}
\newtheorem{corollary}[theorem]{Corollary}

\theoremstyle{definition}

\newtheorem{remark}[theorem]{Remark}

\newcommand{\rr}{{\mathbb R}}
\newcommand{\RR}{{\mathbb R}}
\newcommand{\cc}{{\mathbb C}}
\newcommand{\CC}{{\mathbb C}}
\newcommand{\nn}{{\mathbb N}}
\newcommand{\NN}{{\mathbb N}}

\newcommand{\cD}{{\mathcal D}}
\newcommand{\cS}{{\mathcal S}}
\newcommand{\cN}{{\mathcal N}}

\DeclareMathOperator{\dom}{dom}
\DeclareMathOperator{\spec}{spec}
\DeclareMathOperator{\supp}{supp}
\DeclareMathOperator{\ran}{ran}

\newcommand{\dd}{\,\mathrm{d}}

\renewcommand{\Im}{\operatorname{Im}}

\newcommand{\specd}{\spec_\mathrm{disc}}
\newcommand{\spece}{\spec_\mathrm{ess}}
\newcommand{\specp}{\spec_\mathrm{p}}

\renewcommand{\S}{\Sigma}

\newcommand{\ddd}{\,\mathrm{d}}

\newcommand{\pz}{\partial_{\bar{z}}}
\newcommand{\pzb}{\partial_z}

\begin{document}
	\title[]{On Schr\"odinger operators with oblique transmission conditions on non-smooth curves}

	\author[B. Benhellal]{Badreddine Benhellal}
	\address{Carl von Ossietzky Universit\"at Oldenburg, Fakult\"at V -- Mathematik und Naturwissenschaften, Institut f\"ur Mathematik, Ammer\-l\"an\-der Heerstr. 114--118, 26129 Oldenburg, Germany}
	\email{badreddine.benhellal@uol.de}
	\author[M.Camarasa]{Miguel Camarasa}
	\address{BCAM - Basque Center for Applied Mathematics, Alameda de Maza\-rredo 14, 48009 Bilbao, Spain}
	\email{mcamarasa@bcamath.org}
	\author[K. Pankrashkin]{Konstantin Pankrashkin}
	\address{Carl von Ossietzky Universit\"at Oldenburg, Fakult\"at V -- Mathematik und Naturwissenschaften, Institut f\"ur Mathematik, Ammer\-l\"an\-der Heerstr. 114--118, 26129 Oldenburg, Germany}	\email{konstantin.pankrashkin@uol.de}

	%
	%
	%
	%
	%
	%
	
	\begin{abstract}
		In a recent paper Behrndt, Holzmann, and Stenzel introduced a new class of two-dimensional Schr\"odinger operators with oblique transmissions along smooth curves. We extend most components of this analysis to the case of Lipschitz curves.
	\end{abstract}
	
	\maketitle

	\section{Introduction}
	
	The recent paper \cite{bhs} by Behrndt, Holzmann, and Stenzel introduced a new class of two-dimensional Schr\"odinger operators with interactions supported by curves. The case of smooth curves was studied in detail, in particular, it was shown that these operators arise as a kind of non-relativistic limit of Dirac operators, and some results on the dependence of the eigenvalues on the coupling constant were obtained. In the present work, we show that a significant part of the qualitative analysis can be extended to Lipschitz curves as well, but, at the same time,
	the non-smoothness may lead to an asymptotic behavior of eigenvalues that is different from that obtained in \cite{bhs}.
	
	Let $\Omega_+\subset \RR^2$ be a bounded simply connected domain with Lipschitz boundary $\Sigma$.  Set $\Omega_-:=\mathbb{R}^2\setminus \overline{\Omega_+}$ and denote by $N=(n_1,n_2)$ the unit normal on $\Sigma$ pointing to $\Omega_-$. It will be convenient to use its complexification
	\[
	n:=n_1+in_2:\Sigma\to\CC.
	\]
	A function $f\in L^2(\rr^2)$ will be identified with the pair $(f_+,f_-)$, where $f_{\pm}$ is the restriction of $f$ on $\Omega_\pm$. This gives rise to the identifications \[
	L^2(\rr^2)\simeq L^2(\Omega_+)\oplus L^2(\Omega_-),
	\quad
	H^s(\RR^2\setminus\Sigma)\simeq H^s(\Omega_+)\oplus H^s(\Omega_-), \quad s\ge 0,
	\]
	where $H^s$ stands for the Sobolev space of order $s$, and similar notations and identifications will be used for vector-values functions as well.
	We further denote by $\gamma^\pm$
	the Dirichlet trace operator from $\Omega_\pm$ on $\Sigma$, i.e.
	for functions $f_\pm\in C^\infty(\overline{\Omega_\pm})$ with compact support one has
	$\gamma^\pm f_\pm:=f_\pm|_{\Sigma}$,
	which extends by density to a bounded linear map between suitable Sobolev-type spaces on $\Omega_\pm$ and $\Sigma$ (as discussed in greater detail below).
	
	We are interested in the Laplacian on $\RR^2$ with the so-called \emph{oblique transmission condition} $P_\alpha f=0$ on $\Sigma$, where
	\begin{equation}\label{bc}
		P_\alpha f:=n(\gamma^+f_+-\gamma^-f_- ) +\alpha(\gamma^+\pz f_+ +\gamma^-\pz f_-),
	\end{equation}
	the expression
	\[
	\pz:=\frac{1}{2}(\partial_1+i\partial_2)
	\]
	is known as the Wirtinger derivative and $\alpha\in\RR$ is a parameter. Our goal is to construct a self-adjoint realization of the above operator and to understand its spectral properties and the dependence on $\alpha$.
	
	
	The paper \cite{bhs} was dedicated to the case when
	\begin{equation}
		\label{eqsmooth}
		\Sigma \text{ is $C^\infty$-smooth,}
	\end{equation}
	so let us review the available results first. Assume \eqref{eqsmooth}, then
	\[
	\gamma^\pm:H^1(\Omega_\pm)\to H^\half(\Sigma)
	\]
	are bounded linear maps. Denote by $\Hat H_\alpha$ the linear operator in $L^2(\RR^n)$ acting as
	\[
	\Hat H_\alpha (f_+,f_-)=(-\Delta f_+,-\Delta f_-)
	\]
	on the domain
	\begin{align*}
		\dom \Hat H_\alpha:=\big\{&
		(f_+,f_-)\in H^1(\Omega_+)\oplus H^1(\Omega_-):\\
		&\ (\pz f_+,\pz f_-)\in H^1(\RR^2),
		\  P_\alpha f=0
		\big\},
	\end{align*}	
	then from  \cite[Theorem 1.1]{bhs} it is known that $\Hat H_\alpha$ is self-adjoint
	with $\spece \Hat H_\alpha=[0,\infty)$. The discrete spectrum is empty for $\alpha\ge 0$ and infinite and unbounded from below for $\alpha<0$, without accumulation at $0$, and for any fixed $n\in\NN$ the $n$-th discrete eigenvalue $\lambda_n(\Hat H_\alpha)$ (if counted with multiplicities in the non-increasing order) satisfies
	\begin{equation}
		\label{enta}
		\lambda_n(\Hat H_\alpha)=-\dfrac{4}{\alpha^2}+O(1) \text{ for } \alpha\to 0^-.
	\end{equation}
	In addition, the operator $\Hat H_\alpha$ can be obtained as a limit of Dirac operators. Namely, for $c,\eta,\tau\in\RR$, $c\ne 0$, denote by
	$\Hat B_{c,\eta,\tau}$ the operator in $L^2(\RR^2,\CC^2)$ acting as
	\begin{align*}
		\Hat B_{c,\eta,\tau} (f_+,f_-)&:=(D_cf_+,D_cf_-),\\
		D_c&:=-ic\begin{pmatrix}
			0 & \partial_1-i\partial_2\\
			\partial_1+i\partial_2 & 0
		\end{pmatrix}
		+\dfrac{c^2}{2}\begin{pmatrix}
			1 & 0\\
			0 & -1
		\end{pmatrix},
	\end{align*}
	on the domain
	\begin{align*}
		&\dom \Hat B_{c,\eta,\tau}:=\bigg\{(f_+,f_-)\in H^1(\Omega_+,\mathbb{C}^2)\oplus H^1(\Omega_-,\mathbb{C}^2):\\
		&\quad ic\begin{spmatrix} 0 &\Bar n\\
			n & 0
		\end{spmatrix}(\gamma^+f_+-\gamma^- f_-)+\frac{1}{2}\begin{spmatrix}
			\eta +\tau & 0 \\
			0 & \eta-\tau
		\end{spmatrix}(\gamma^+f_+ +\gamma^- f_-)=0
		\bigg\},
	\end{align*}
	then \cite[Theorem 1.2]{bhs} states that for any fixed $\alpha\in\RR$ and $\lambda\in\CC\setminus\RR$ one has the norm resolvent convergence
	\[
	\left\|\Big(\Hat B_{c,-\frac{\alpha c^2}{2},\frac{\alpha c^2}{2}}-\big(\lambda+\tfrac{c^2}{2}\big)\Big)^{-1}-\begin{pmatrix}
		(\Hat H_\alpha-\lambda)^{-1} & 0 \\
		0 & 0
	\end{pmatrix}\right\|=O\Big(\frac{1}{c}\Big), \ c\to \infty.
	\]
	Our objective is to obtain similar results when the curve $\Sigma$ is non-smooth. We will show that the basic properties
	as the self-adjointness and the resolvent convergence can be adapted by a suitable extension of operator domains. On the other hand,
	the asymptotics \eqref{enta} turns out to be false in general.

	{}From now we assume:
	\begin{equation}
		\label{eqlip}
		\text{the curve $\Sigma$ is Lipschitz.}
	\end{equation}
	For $s\ge 0$ and open $\Omega\subset\RR^2$ define the space
	\[
	H^{s}_{\Delta}(\Omega)=\big\{g\in H^{s}(\Omega):\, \Delta g\in L^2(\Omega)\big\},
	\]
	which is a Hilbert space with the scalar product
	\[
	\langle g,\Tilde g \rangle_{H^{s}_{\Delta}(\Omega)}:=\langle g,\Tilde g\rangle_{H^s(\Omega)}+ \langle \Delta g,\Delta \Tilde g\rangle_{L^2(\Omega)}.
	\]
	Recall that the Dirichlet trace maps $\gamma^\pm$ can be viewed as bounded linear maps from $H^s_{\Delta}(\Omega_{\pm})$ to $H^{s-\half}(\Sigma)$ for any $s\in[\frac{1}{2},\frac{3}{2}]$, see \cite[Corollary 3.7]{bgm}.
	
	For $\alpha\in\RR$ denote by $H_\alpha$ the linear operator in $L^2(\RR^2)$ acting as
	\[
	H_\alpha (f_+,f_-)= (-\Delta f_+,\,-\Delta f_-)
	\]
	on the domain 
	\begin{equation}
		\label{def:schrodinger}
		\begin{aligned}
			\dom H_\alpha:=\big\{&
			(f_+,f_-)\in H^\half_{\Delta}(\Omega_+)\oplus H^\half_{\Delta}(\Omega_-):\\
			&\ (\pz f_+,\pz f_-)\in H^1(\RR^2),
			\  P_\alpha f=0
			\big\}.
		\end{aligned}	
	\end{equation}
	In Section \ref{sec2} we show that $H_\alpha$ is self-adjoint, its essential spectrum is the positive half-line. The discrete spectrum is empty for $\alpha>0$ and infinite (with accumulation at $-\infty$ only) for $\alpha<0$. This part of the discussion represents an adaptation of the analysis of \cite{bhs} to the Lipschitz case with the help of recent developments from \cite{bgm,bpz}. In Section \ref{sec3} we discuss the behavior of the eigenvalues of $H_\alpha$ for $\alpha\to 0^-$. Similarly to \cite{bhs} we use a relation with $\delta$-potentials, but due to the lower regularity
	of $\Sigma$ only much weaker assumptions can be used. As a result, we show the two-sided estimate only: for any $n\in\NN$ there are $0<A<B$ such that
	\[
	-\dfrac{B}{\alpha^2}\le 
	\lambda_n(H_\alpha)\le-\dfrac{A}{\alpha^2} \text{ for }\alpha\to 0^-.
	\]
	
	The asymptotic spectral analysis of Schr\"odinger operators with $\delta$-potentials supported on non-smooth curves clearly has its own interest. Due to an expected large number of technicalities, we prefer to discuss it in detail somewhere else. In the present work, we restrict our attention to the first eigenvalue and a special geometry to show that the asymptotic behavior of eigenvalues can be different from the one in the smooth case. More precisely, we show in Section \ref{sec4} that for any $B\in(1,4)$ there exists a non-smooth curve $\Sigma$ such that for the associated operator $H_\alpha$
	there holds
	\[
	\lambda_1(H_\alpha)=-\dfrac{B}{\alpha^2}+o\Big(\dfrac{1}{\alpha^2}\Big) \text{ for } \alpha\to 0^-,
	\]
	which is clearly different from \eqref{enta}. The analysis in Sections \ref{sec3} and \ref{sec4} is based on the min-max principle
	for the eigenvalues.

	\section{Self-adjointness and basic spectral properties}\label{sec2}
	
	In addition to the Dirichlet trace maps $\gamma^\pm$ we will consider the Neumann trace maps $\gamma^\pm_n$ defined
	for $f_\pm\in C^\infty(\overline{\Omega_\pm})$ by $\gamma^\pm_n f_\pm:= \pm \langle N,f_\pm|_\Sigma\rangle$
	and then extended by continuity in suitable function spaces. In particular, $\gamma^\pm_n:H^s_\Delta(\Omega_\pm)\to H^{s-\frac{3}{2}}(\Sigma)$ are bounded
	for any $s\in[\frac{1}{2},\frac{3}{2}]$, see \cite[Corollary 5.7]{bgm}
	
	Recall that for any $\lambda\in\mathbb{C}\setminus[0,\infty)$ the function
	\[
	\Phi(x)=\frac{1}{2\pi}K_0(-i\sqrt{\lambda}|x|), \quad x\in\rr^2\setminus\{0\},
	\]
	is a fundamental solution of $-\Delta -\lambda$. Here $K_0$ is the modified Bessel function of the second kind and the convention $\Im\sqrt{\lambda}>0$ is used.
	
	For $\lambda\in\mathbb{C}\setminus[0,\infty)$ let $\cS(\lambda):L^2(\Sigma)\rightarrow L^{2}(\mathbb{R}^2)$ be the single layer potential given by
	\begin{equation}\label{singlelayer}
		\cS(\lambda)g(x)=\int_{\Sigma}\Phi(x-y)g(y)\dd\sigma(y), \ g\in L^2(\Sigma), \ x\in\mathbb{R}^2\setminus\Sigma,
	\end{equation}
	and $S(\lambda):\,L^2(\Sigma)\rightarrow L^2(\Sigma)$ be the single layer boundary integral operator given by
	\begin{equation}\label{singlelayer2}
		S(\lambda)g(x)=\int_{\Sigma}\Phi(x-y)g(y)\dd\sigma(y), \ g\in L^2(\Sigma), \ x\in\Sigma.
	\end{equation}
	One has by construction $(-\Delta-\lambda)\cS(\lambda)=0$ in $\rr^2\setminus\Sigma$, and it is well known that
	\[
	\cS(\lambda):L^2(\Sigma)\rightarrow H^{1}(\mathbb{R}^2),\quad S(\lambda):L^2(\Sigma)\rightarrow H^1(\Sigma),
	\]
	are well-defined and bounded, and there holds $\big(\cS(\lambda)g\big)_\pm\in H^{\frac{3}{2}}_\Delta(\Omega_\pm)$ and
	\begin{align}\label{SLprop}
		\gamma^\pm \cS(\lambda)g= S(\lambda)g \quad \text{and}\quad \gamma^+_n \cS(\lambda)g+\gamma^-_n \cS(\lambda)g=g,
	\end{align}
	for all $g\in L^2(\Sigma)$, see \cite[Theorem 1]{cost2} and \cite[Theorem 6.11 and Eq. (7.5)]{Mc}.
	An integration by parts gives	
	\begin{align*}
		\lambda \big\|\cS(\lambda)g\big\|^2_{L^2(\Omega_\pm)}&=\big\langle \cS(\lambda)g,-\Delta \cS(\lambda) g\big\rangle_{L^2(\Omega_\pm)}\\
		&=	\big\|\nabla \cS(\lambda) g \|^2_{L^2(\Omega_\pm)}
		- \big\langle \gamma^\pm \cS(\lambda) g, \gamma^\pm_n \cS(\lambda) g\big\rangle_{L^2(\S)},
	\end{align*} 
	and by \eqref{SLprop} it follows that
	\begin{equation}
		\label{ssll}
		\begin{aligned}
			\big\langle S(\lambda)g,g  \big\rangle_{L^2(\Sigma)} &
			=  \big\|\nabla \cS(\lambda) g \|^2_{L^2(\Omega_+)}+ \big\|\nabla \cS(\lambda) g \|^2_{L^2(\Omega_-)}\\
			&\quad - \lambda \big\|\cS(\lambda)g\big\|^2_{L^2(\RR^2)},\ g\in L^2(\Sigma).
		\end{aligned}
	\end{equation}	
	The operator $S(\lambda):L^2(\Sigma)\to L^2(\Sigma)$ is compact (as it has a Hilbert-Schmidt integral kernel), and for any $\lambda\in(-\infty,0)$ it is self-adjoint and non-negative. We further remark that if $S(\lambda)g=0$ for some $\lambda<0$ and $g\in L^2(\Sigma)$, then \eqref{ssll} shows
	that $\cS(\lambda)g=0$ in $\Omega_\pm$, and the second identity in \eqref{SLprop} shows that $g=0$. This means that $S(\lambda)$
	is injective for all $\lambda<0$.
	
	Let $H$ be the free Laplacian in $L^2(\RR^2)$,
	\[
	\dom H:=H^2(\RR^2),\quad H f:=-\Delta f.
	\]
	For any $f\in L^2(\mathbb{R}^2)$ and $\lambda\in\CC\setminus[0,\infty)$ we have $(H-\lambda)^{-1}f\in H^2(\RR^2)$, therefore,
	\begin{align*}
		\gamma^+\big[(H-\lambda)^{-1}f\big]_+ &=\gamma^-\big[(H-\lambda)^{-1}f\big]_-=\gamma (H-\lambda)^{-1}f,\\
		\gamma^+ \big[\pz (H-\lambda)^{-1}f\big]_+&=\gamma^- \big[\pz (H-\lambda)^{-1}f\big]_-=\gamma \pz (H-\lambda)^{-1}f,
	\end{align*}
	where
	\[
	\gamma:\, H^1(\RR^2)\to H^\half(\Sigma),\quad f\mapsto f|_\Sigma,
	\]
	is the standard Sobolev trace, and for the boundary operator $P_\alpha$ from \eqref{bc} we have then
	\begin{equation}
		\label{paf}
		P_{\alpha}(H-\lambda)^{-1}f=2\alpha \gamma \pz (H-\lambda)^{-1}f.
	\end{equation}
	Let us study in greater detail the expression on the right-hand side.
	
	\begin{lemma}\label{lemmapsilambda}
		For any $\lambda\in \cc\setminus[0,\infty)$ define the operator $\Psi_\lambda$ as follows:
		\begin{equation}
			\label{Llambda}
			\Psi_\lambda:L^2(\RR^2)\to L^2(\Sigma),\quad 
			\Psi_{\lambda}f:=2\gamma\pz (H-\lambda)^{-1}f,
		\end{equation}
		then:
		\begin{itemize}
			\item[(a)] $\Psi_\lambda:L^2(\rr^2)\rightarrow L^2(\Sigma)$ is compact with $\ran \Psi_{\lambda}\subseteq H^\half (\Sigma)$, 
			\item[(b)] $\Psi_{\lambda}^*=-2\pzb\cS (\bar{\lambda})$ with $\pzb:=\frac{1}{2}(\partial_1-i\partial_2)$,
			\item[(c)] for any $\varphi\in L^2(\Sigma)$ there holds
			\begin{align}
				\label{item2lemmaeq2}
				\gamma^+( \pz \Psi_\lambda^*\varphi)_+ + \gamma^-(\pz \Psi^*_\lambda\varphi)_-&=\bar{\lambda} S(\bar{\lambda})\varphi,\\
				\label{item2lemmaeq1}
				n\Big(\gamma^+ (\Psi^*_\lambda\varphi)_+ - \gamma^-(\Psi^*_\lambda\varphi)_-\Big)&=-\varphi.
			\end{align}
		\end{itemize} 
		Furthermore, if 
		\begin{equation}
			\label{Hlambda}
			\begin{aligned}
				\cN_{\lambda}:=\big\{f\in H^\half(\rr^2\setminus\Sigma):&\    (\pz f_+,\pz f_-) \in H^{1}(\rr^2),\\
				&\  (-\Delta -\bar{\lambda})f=0 \text{ in } \rr^2\setminus\Sigma\big\},
			\end{aligned}
		\end{equation}
		then
		\begin{itemize}
			\item[(d)] $\Psi_\lambda^*:L^2(\Sigma)\rightarrow \cN_{\lambda}$ is bijective.
		\end{itemize}
	\end{lemma}

	\begin{proof}
		(a) The map $\Psi_\lambda$ is the double of the composition of the bounded linear operators
		\begin{gather*}
			(H-\lambda)^{-1}:\,L^2(\rr^2)\rightarrow H^2(\rr^2),
			\quad
			\pz:H^2(\RR^2)\to H^1(\RR),\\
			\gamma:H^1(\RR^2)\to H^\half(\Sigma).
		\end{gather*}
		As the embedding $H^\half(\Sigma)\hookrightarrow L^2(\Sigma)$ is compact, the claim follows.

		(b) Let $f\in \cD(\RR^2)$ and $\varphi\in L^2(\Sigma)$, then
		\begin{align*}
			\big\langle -2\pzb\cS &(\bar{\lambda})\varphi,f\big\rangle_{L^2(\RR^2)}=-\int_{\RR^2} \overline{\mathstrut(\partial_1-i\partial_2)\cS(\bar\lambda)\varphi}\,f\,\dd x\\
			&=-\int_{\RR^2} \big((\partial_1+ i\partial_2)\cS(\lambda)\Bar \varphi\big)\,f\,\dd x\\
			&=-\Big( (\partial_1+ i\partial_2)\cS(\lambda)\Bar \varphi,f\Big)_{\cD'(\RR^2),\cD(\RR^2)}\\
			&=\Big( \cS(\lambda)\Bar \varphi,(\partial_1-i\partial_2)f\Big)_{\cD'(\RR^2),\cD(\RR^2)}\\
			&=\Big( \cS(\lambda)\Bar \varphi,2\pz f\Big)_{\cD'(\RR^2),\cD(\RR^2)}\\
			&=\int_{\RR^2} \int_\Sigma \Phi_\lambda(x-y)\overline{\varphi(y)}\, \dd\sigma(y)\, 2\pz f(x)\, \dd x\\
			&=
			\int_\Sigma \overline{\varphi(y)} \int_{\RR^2} \Phi_\lambda(x-y)\, 2\pz f(x)\, \dd x \,\dd \sigma(y)\\
			&=
			\int_\Sigma \overline{\varphi(y)} \,2\gamma (H-\lambda)^{-1}\pz f(y) \,\dd\sigma(y)\\
			&=\big\langle \varphi,2\gamma (H-\lambda)^{-1}\pz f\big\rangle_{L^2(\Sigma)}.
		\end{align*}
		One easily checks that $(H-\lambda)^{-1}\pz f=\pz (H-\lambda)^{-1}f$ (for example, one applies the Fourier transform on both sides), so the previous computations yield
		\[
		\big\langle -2\pzb\cS (\bar{\lambda})\varphi,f\big\rangle_{L^2(\RR^2)}=\big\langle \varphi,\Psi(\lambda) f\big\rangle_{L^2(\Sigma)}.
		\]
		This extends by density to all $f\in L^2(\RR^2)$ and implies the sought identity.

		(c)	By \cite[Eq (2.127)]{GM} there holds $\ran \cS (\lambda)\subset H^{\frac{3}{2}}_\Delta(\RR^2\setminus\Sigma)$, which yields $\ran \Psi_{\lambda}^*\subset H^\half (\RR^2\setminus\Sigma)$. For any $\varphi\in L^2(\Sigma)$ one has $(-\Delta -\lambda)\cS (\lambda)=0$ in $\RR^2\setminus\Sigma$, and due to $\pz\pzb=\frac{1}{4}\Delta$ we obtain
		\begin{equation}\label{eqdzbar}
			\begin{aligned}
				\big(\partial_{\bar{z}}\Psi^*_{\lambda}\varphi\big)_{\pm}&=\Big(-\frac{1}{2}\Delta \cS(\bar {\lambda})\varphi\Big)_{\pm}\\
				&=\frac{\bar{\lambda}}{2} \big(\cS(\bar{\lambda})\varphi\big)_{\pm} \in H^{\frac{3}{2}}(\Omega_{\pm})\subset H^1(\Omega_{\pm}),
			\end{aligned}
		\end{equation}
		We have $\cS(\Bar{\lambda})\varphi\in H^1(\RR^2)$, hence,
		\[
		\gamma^+\big(\cS(\Bar \lambda)\varphi\big)_+=\gamma^-\big(\cS(\Bar \lambda)\varphi\big)_-=\gamma S(\bar\lambda)\varphi,
		\]
		and \eqref{eqdzbar} gives
		\[
		\gamma^\pm \big(\partial_{\bar{z}}\Psi^*_{\lambda}\varphi\big)_{\pm}= \frac{\bar{\lambda}}{2} \gamma \cS(\bar{\lambda})\varphi=\frac{\bar{\lambda}}{2} S(\bar{\lambda})\varphi.
		\]
		This shows the identity \eqref{item2lemmaeq2} as well as the inclusion
		\begin{equation}
			\label{lokal00}
			\Big(\big(\partial_{\bar{z}}\Psi^*_{\lambda}\varphi\big)_+, \big(\partial_{\bar{z}}\Psi^*_{\lambda}\varphi\big)_-\Big)\in H^1(\RR^2).
		\end{equation}
		
		Let $f=(f_+,f_-)\in H^1(\RR^2)$ with $f_\pm\in H^{\frac{3}{2}}_\Delta(\Omega_\pm)$ and $g=(g_+,g_-)\in C^\infty_0(\RR^2)$, then
		the integration by parts gives
		\begin{align*}
			\pm \int_\Sigma &n \gamma^\pm(2\pzb f_\pm)\gamma g\, \dd\sigma=\int_{\Omega_\pm} 2\pz\big((2 \pzb f_\pm ) g_\pm \big)\, \dd x\\
			&=\int_{\Omega_\pm} (4\pz\pzb f_\pm)g_\pm\, \dd x +\int_{\Omega_\pm} (2\pzb f_\pm ) (2\pz g_\pm)\, \dd x\\
			&=\int_{\Omega_\pm} (\Delta f_\pm)g_\pm\, \dd  x +\int_{\Omega_\pm} 4\pzb (f_\pm\,\pz g_\pm)\, \dd x-\int_{\Omega_\pm} f_\pm\,4\pzb\pz g_\pm\, \dd x\\
			&=\int_{\Omega_\pm} (\Delta f_\pm)g_\pm\, \dd x \pm \int_{\Sigma} 2\Bar n \gamma^\pm f_\pm\,\gamma \pz g\, \dd\sigma-\int_{\Omega_\pm} f_\pm\Delta g_\pm\, \dd x\\
			&=\int_{\Sigma} (\gamma^\pm_n f_+) \gamma g\, \dd\sigma \pm \int_{\Sigma} 2\Bar n \gamma^\pm f_\pm\,\gamma \pz g\, \dd\sigma-\int_{\Sigma} \gamma^\pm f_\pm\,\gamma^\pm_n g_\pm\, \dd\sigma.
		\end{align*}	
		We have
		\[
		\gamma^+f_+=\gamma^- f_-,\quad \gamma^+_n g_+ + \gamma^-_n g_-=0,
		\]
		which yields (with $h:=\gamma g$)
		\[
		\int_\Sigma n \big(\gamma^+(2\pzb f_+)- \gamma^-(2\pzb f_-)\big)\,h\,\dd\sigma=
		\int_{\Sigma} (\gamma^+_n f_+ + \gamma^+_n f_-) h\, \dd\sigma.
		\]
		By density this extends to all $g\in H^1(\RR^2)$, hence, to all $h\in H^\half(\Sigma)$, and then for all $h\in L^2(\Sigma)$, and one obtains
		\[
		n \big(\gamma^+(2\pzb f_+)- \gamma^-(2\pzb f_-)\big)=\gamma^+_n f_+ + \gamma^+_n f_-.
		\]
		For $f:=\cS(\Bar\lambda)\varphi$ with $\varphi\in L^2(\Sigma)$ we have $2\pzb f=-\Psi_\lambda^*\varphi$, which results in
		\[
		n\Big(\gamma^+ (\Psi^*_\lambda\varphi)_+ - \gamma^-(\Psi^*_\lambda\varphi)_-\Big)=-\big(\gamma^+_n \cS(\bar\lambda)\varphi+\gamma^-_n \cS(\bar\lambda)\varphi\big)\stackrel{\eqref{SLprop}}{=}-\varphi
		\]
		and proves \eqref{item2lemmaeq1}.
		
		(d) In $\cD'(\RR^2\setminus\Sigma)$ one has
		\begin{align*}
			-\Delta \Psi_{\lambda}^*\varphi&=-2\Delta \pzb \cS(\Bar \lambda)\varphi=-2\pzb \big(\Delta \cS(\Bar \lambda)\varphi\big)\\
			&=-2\big( \pzb \Bar\lambda \cS(\Bar \lambda)\varphi\big)=\Bar \lambda \Psi_\lambda^*\varphi,
		\end{align*}
		i.e. $(-\Delta-\bar\lambda)\Psi_\lambda^*\varphi=0$ in $\RR^2\setminus\Sigma$. Together with \eqref{eqdzbar} and \eqref{lokal00}
		this shows the inclusion $\Psi_\lambda^*\big( L^2(\Sigma)\big)\subset \cN_\lambda$,
		and \eqref{item2lemmaeq1} implies the injectivity of $\Psi_\lambda^*$.
		
		It remains to show $\Psi_\lambda^*\big( L^2(\Sigma)\big)=\cN_\lambda$.
		Let $f=(f_+,f_-)\in \cN_{\lambda}$ and set
		\[
		\varphi:=-n(\gamma^+f_+ - \gamma^-f_-)\in L^2(\Sigma),\quad g:=\Psi_{\lambda}^*\varphi,\quad u:=f-g.
		\]
		We are going to show that $u=0$. One has already  $g\in \cN_\lambda$, hence, $u\in \cN_\lambda\subset H^\half_\Delta(\RR^2\setminus\Sigma)$.
		Using \eqref{item2lemmaeq1} we get
		\[
		n(\gamma^+g_+ - \gamma^-g_-)=-\varphi=n(\gamma^+f_+ - \gamma^-f_-),
		\]
		hence, $n(\gamma^+u_+ - \gamma^-u_-)=0$, and then $\gamma^+u_+ = \gamma^-u_-$.
		Due to $u\in\cN_\lambda$ we have $(\pz u_+,\pz u_-)\in H^1(\rr^2)$, and using the jump formula in $\cD'(\RR^2)$
		we obtain
		\[
		\pz u=( \pz u_+,\pz u_-) + n (\gamma^+u_+-\gamma^-u_-)=( \pz u_+,\pz u_-)\in H^{1}(\RR^2).
		\]
		As $\pz$ is a first-order elliptic operator, due to the elliptic regularity theorem one arrives at $u\in H^2_\mathrm{loc}(\RR^2)$.
		From the inclusion $u\in \cN_\lambda$ we obtain $(-\Delta-\Bar \lambda)u=0$ in $\RR^2\setminus\Sigma$,
		and from $u\in H^2_\mathrm{loc}(\RR^2)$ it follows $(-\Delta-\Bar \lambda)u=0$ in $\RR^2$, so $u\in\dom H$ with  $H u=\Bar \lambda u$,
		and then $u=0$.	
	\end{proof}

	\begin{theorem}\label{thm2} Let $\alpha\in\RR$, then:
		\begin{itemize}
			\item[(a)] For any $\lambda\in\CC\setminus [0,+\infty)$ there holds
			\begin{equation}
				\label{eqbs}
				\ker(H_\alpha-\lambda)= \Psi_{\bar{\lambda}}^*\ker \big(I-\alpha\lambda S(\lambda)\big).
			\end{equation}
			In particular,
			\[
			\lambda\in\specp H_\alpha \iff 0\in \specp \big(I-\alpha\lambda S(\lambda)\big).
			\]
			\item[(b)]  For any $\lambda\in\cc\setminus\rr$ the operator
			\[
			\big(I-\alpha\lambda S(\lambda)\big)^{-1}: L^2(\Sigma)\rightarrow L^2(\Sigma)
			\] is well-defined and bounded,
			and there holds
			\begin{equation}\label{resolvent}
				(H_\alpha-\lambda)^{-1}=(H-\lambda)^{-1}+\alpha\Psi^*_{\bar\lambda}\big(I-\alpha\lambda S(\lambda)\big)^{-1}\Psi_{\lambda}.
			\end{equation}
			\item[(c)] The operator $H_\alpha$ is self-adjoint.
		\end{itemize}
	\end{theorem}

	\begin{proof}
		Due to $C^\infty_0 (\rr^2\setminus\Sigma)\subset \dom H_\alpha$ the operator $H_\alpha$ is densely defined.
		Let us show that it is also symmetric. Let $f\in\dom H_\alpha$, then
		\begin{align*}
			\langle H_\alpha f,f\rangle_{L^2(\rr^2)}&= \langle -\Delta f_{+},f_{+}\rangle_{L^2(\Omega_{+})}+ \langle -\Delta f_{-},f_{-}\rangle_{L^2(\Omega_{-})}.
		\end{align*}
		Using the integration by parts we obtain
		\begin{align*}
			\langle -\Delta f_\pm,f_\pm\rangle_{L^2(\Omega_{\pm})}&=
			\langle -4\pzb \pz f_{\pm},f_{\pm}\rangle_{L^2(\Omega_{\pm})}\\
			&=
			4\langle \pz f_{\pm},\pz f_{\pm} \rangle_{L^2(\Omega_{\pm})}\mp 
			2 \langle \gamma^{\pm} \pz f_{\pm}, n\gamma^{\pm}f_{\pm} \rangle_{L^2(\Sigma)}.
		\end{align*}
		Due to $(\pz f_+,\pz f_-)\in H^1(\RR^2)$ one has $\gamma^{+}\pz f_{+}=\gamma^{-}\pz f_{-}$, hence,
		\begin{align*}
			\langle H_\alpha f,f\rangle_{L^2(\rr^2)}=&4\|(\partial_{\bar{z}}f_{+},\partial_{\bar{z}}f_{-})\|^2_{L^2(\rr^2)}\\
			& -\big\langle \gamma^{+}\pz f_{+}+ \gamma^{-}\pz f_{-}, n(\gamma^{+}f_{+}-\gamma^{-}f_{-}) \big\rangle_{L^2(\Sigma)}.
		\end{align*}
		The transmission condition $P_\alpha f=0$ implies
		\[
		n(\gamma^+f_+-\gamma^-f_- )=-\alpha(\gamma^+\pz f_+ +\gamma^-\pz f_-),
		\]
		so one arrives at
		\begin{equation}
			\label{qfha}
			\begin{aligned}
				\langle H_\alpha f,f\rangle_{L^2(\rr^2)}&=4\big\|(\pz f_+,\pz f_-)\big\|^2_{L^2(\mathbb{R}^2)}\\
				&\quad +\alpha\| \gamma^{+}\pz f_{+} + \gamma^{-}\pz f_{-}\|_{L^2(\Sigma)}^2,
			\end{aligned}
		\end{equation}
		and the right-hand side is obviously real-valued, which shows the symmetry of $H_\alpha$.

		(a) Let $\lambda\in\mathbb{C}\setminus [0,+\infty)$. We remark first that
		\[
		\ker(H_\alpha-\lambda)=\big\{ f\in \cN_{\bar\lambda}:\, P_\alpha f=0\big\},
		\]
		with $\cN_\lambda$ from \eqref{Hlambda}.
		For each function  $f\in \cN_{\bar\lambda}$ there is a unique $\varphi\in L^2(\Sigma)$ with $f=\Psi_{\bar{\lambda}}^*\varphi$, see Lemma \ref{lemmapsilambda}(d). By Lemma \ref{lemmapsilambda}(c) there holds
		\begin{align*}
			P_\alpha f&\equiv n(\gamma^+f_+-\gamma^-f_- ) +\alpha(\gamma^+\pz f_+ +\gamma^-\pz f_-)\\
			&=-\varphi+\alpha \lambda S(\lambda)\varphi=-\big(I-\alpha\lambda S(\lambda)\big)\varphi,
		\end{align*}
		showing that $P_\alpha f=0$ if and only if $\varphi \in\ker \big(I-\alpha\lambda S(\lambda)\big)$.

		(b) Let $\lambda\in\CC\setminus\RR$. As  $H_\alpha$ is symmetric, there holds $\ker(H_\alpha-\lambda)=\{0\}$, and then $I-\alpha\lambda S(\lambda)$ is  injective by (a). As $S(\lambda)$ is compact, the operator $I-\alpha\lambda S(\lambda)$ is also surjective due to Fredholm alternative,
		so its inverse is well-defined and bounded.
		
		Let $f\in L^2(\RR^2)$, then $\big(I-\alpha\lambda S(\lambda)\big)^{-1}\Psi_{\lambda}f\in L^2(\Sigma)$ is well-defined, so consider the function
		\begin{equation}\label{resolventformula}
			g:=(H-\lambda)^{-1}f+\alpha \Psi_{\bar{\lambda}}^*\big(I-\alpha\lambda S(\lambda)\big)^{-1}\Psi_{\lambda}f.
		\end{equation}
		We are going to show that $g\in \dom H_\alpha$.
		The first summand on the right-hand side of \eqref{resolventformula} is in $H^2(\RR^2)$ and the second one
		is in $\cN_{\bar\lambda}$, which yields
		\[
		g\in H^\half_{\Delta}(\RR^2\setminus\Sigma), \quad (\pz g_+,\pz g_-)\in H^1(\RR^2).
		\]
		Using Lemma \ref{lemmapsilambda}(c) we obtain
		\begin{align*}
			P_{\alpha}&\Psi_{\bar{\lambda}}^*\big(I-\alpha\lambda S(\lambda)\big)^{-1}\Psi_{\lambda}f\\
			&=-\big(I-\alpha\lambda S(\lambda)\big)^{-1}\Psi_{\lambda}f+\alpha\lambda S(\lambda)\big(I-\alpha\lambda S(\lambda)\big)^{-1}\Psi_{\lambda}f \\
			&= -\big(I-\alpha\lambda S(\lambda)\big)\big(I-\alpha\lambda S(\lambda)\big)^{-1}\Psi_{\lambda}f= -\Psi_{\lambda}f,
		\end{align*}
		and by \eqref{paf} we have $P_{\alpha}(H-\lambda)^{-1}f=\alpha \Psi_\lambda f$, therefore,
		\begin{align*}
			P_\alpha g&=P_{\alpha}(H-\lambda)^{-1}f+\alpha \Psi_{\bar{\lambda}}^*\big(I-\alpha\lambda S(\lambda)\big)^{-1}\Psi_{\lambda}f\\
			&=-\alpha \Psi_\lambda f+\alpha \Psi_\lambda f=0,
		\end{align*}
		which shows $g\in \dom H_\alpha$. Finally, in $\cD'(\RR^2\setminus\Sigma)$ there holds
		\begin{align*}
			(-\Delta-\lambda)g&=(-\Delta-\lambda)(H-\lambda)^{-1}f\\
			&\quad+\alpha (-\Delta-\lambda)\Psi_{\bar{\lambda}}^*\big(I-\alpha\lambda S(\lambda)\big)^{-1}\Psi_{\lambda}f,
		\end{align*}	
		and using $(-\Delta-\lambda)\Psi_{\bar{\lambda}}^*=0$, see Lemma \ref{lemmapsilambda}(d), and
		\[
		(-\Delta-\lambda)(H-\lambda)^{-1}f=(H-\lambda)(H-\lambda)^{-1}f=f,
		\]	
		we arrive at $(-\Delta-\lambda)g=f$. This shows $g=(H_\alpha-\lambda)^{-1}f$.
		
		(c) We have seen that $H_\alpha$ is symmetric, and due to (b) there holds $\spec H_\alpha\subset\RR$, which shows the self-adjointness.
	\end{proof}
	
	\begin{remark} With the resolvent formula \eqref{resolvent}, many results obtained in \cite{bhs}
		for smooth $\Sigma$ can be transferred to the case of Lipschitz $\Sigma$. The resolvent formula has literally the same form as in \cite{bhs},
		and most results used for $S(\lambda)$ hold for Lipschitz $\Sigma$ too, as they are mainly based on $L^2$-estimates for the integral kernels
		and do not require the smoothness of $\Sigma$. We will address some of the results in a more precise way.
	\end{remark}
	
	\begin{remark}
		For $c,\eta,\tau\in\RR$, $c\ne 0$, denote by
		$B_{c,\eta,\tau}$ the operator in $L^2(\RR^2,\CC^2)$ acting as
		\[
		B_{c,\eta,\tau} (f_+,f_-):=(D_cf_+,D_cf_-)
		\]
		on the domain
		\begin{align*}
			\dom\, B_{c,\eta,\tau}:=\bigg\{&(f_+,f_-)\in H^\half(\Omega_+, \mathbb{C}^2)\oplus H^\half(\Omega_-,\mathbb{C}^2):\\
			&(D_cf_+,D_c f_-)\in L^2(\Omega_+,\mathbb{C}^2)\oplus L^2(\Omega_-,\mathbb{C}^2),\\
			ic\begin{spmatrix} 0 &\Bar n\\
				n & 0
			\end{spmatrix}(\gamma^+f_+&-\gamma^- f_-)+\frac{1}{2}\begin{spmatrix}
				\eta +\tau & 0 \\
				0 & \eta-\tau
			\end{spmatrix}(\gamma^+f_+ +\gamma^- f_-)=0
			\bigg\},
		\end{align*}
		then \cite[Theorem 3.1]{bpz} states that $B_{c,-\eta,\eta}$ is self-adjoint for any $c\ne 0$ and $\eta\in\RR$
		and derives a resolvent formula. Namely, if one denotes by $B_c$ the free Dirac operator in $L^2(\RR^2,\CC^2)$,
		\[
		\dom B_c=H^1(\RR^2,\CC^2),\quad B_cf:=D_cf,
		\]
		then for any $\lambda\in\CC\setminus\RR$ one has the resolvent formula
		\begin{multline*}
			\Big(B_{c,-\frac{\alpha c^2}{2},\frac{\alpha c^2}{2}}-\big(\lambda+\tfrac{c^2}{2}\big)\Big)^{-1}=\Big(B_c-\big(\lambda+\tfrac{c^2}{2}\big)\Big)^{-1}\\
			+\Phi_{\lambda+\frac{c^2}{2}} \begin{pmatrix}
				0 & 0 \\
				0 & \big(I-\alpha c^2 \Theta_{\lambda+\frac{c^2}{2}}\big)^{-1}
			\end{pmatrix}
			\Phi_{\Bar\lambda+\frac{c^2}{2}}^*,
		\end{multline*}
		where
		\begin{align*}
			\Phi_{\lambda}:&\  L^2(\Sigma,\cc^{2})  \longrightarrow L^2(\rr^2,\cc^2),\\
			\Phi_{\lambda}g(x) &=\int_{\Sigma} G_\lambda(x-y)g(y)\,\dd\sigma(y), \quad  x\in\rr^2\setminus\Sigma,\\
			G_\lambda (x)&:= \dfrac{1}{2\pi c} K_{0}\Big( -i\sqrt{\tfrac{\lambda^2}{c^2}-\tfrac{c^2}{4}}|x| \Big)
			\begin{spmatrix}
				\frac{\lambda}{c} +\frac{c}{2} & 0 \\
				0 & \frac{\lambda}{c} -\frac{c}{2}
			\end{spmatrix}\\
			&\quad+ 
			\dfrac{1}{2\pi c |x|} \sqrt{\tfrac{\lambda^2}{c^2}-\tfrac{c^2}{4}}
			K_{1}\Big( -i\sqrt{\tfrac{\lambda^2}{c^2}-\tfrac{c^2}{4}}|x| \Big)
			\begin{spmatrix}
				0 & x_1-ix_2\\
				x_1+i x_2 & 0
			\end{spmatrix},\\
			\Theta(\lambda)&:=\tfrac{1}{c}\big(\tfrac{\lambda}{c}-\tfrac{c}{2}\big)S\big({\sqrt{\tfrac{\lambda^2}{c^2}-\tfrac{c^2}{4}}}\big).
		\end{align*}
		With the above resolvent formulas one can use the same computation as in \cite[Theorem 1.2]{bhs}, which only involves $L^2$-estimates
		for various integral operators in $L^2(\Sigma)$, to show that
		for any fixed $\alpha\in\RR$ and $\lambda\in\CC\setminus\RR$ one has
		\[
		\left\|\Big(B_{c,-\frac{\alpha c^2}{2},\frac{\alpha c^2}{2}}-\big(\lambda+\tfrac{c^2}{2}\big)\Big)^{-1}-\begin{pmatrix}
			(H_\alpha-\lambda)^{-1} & 0 \\
			0 & 0
		\end{pmatrix}\right\|=O\Big(\frac{1}{c}\Big),\ c\to \infty.\]
	\end{remark}
	
	\begin{remark}\label{rmk9}
		Similarly one derives the basic spectral properties of $H_\alpha$. First, due to the compactness
		of $\Psi_\lambda$, see Lemma \ref{lemmapsilambda}(a) and the resolvent formula \eqref{resolvent}
		we conclude that $(H_\alpha-\lambda)^{-1}-(H-\lambda)^{-1}$ is compact for any $\lambda\in\CC\setminus\RR$,
		which implies
		\begin{equation}
			\label{sess}
			\spece H_\alpha=\spece H=[0,\infty) \text{ for any }\alpha\in\RR.
		\end{equation}
		For any $\alpha\ge 0$ and any $f\in\dom H_\alpha$ one has $\langle H_\alpha f,f\rangle_{L^2(\RR^2)}\ge 0$ due to \eqref{qfha}, which gives
		$\spec H_\alpha\subset[0,\infty)$, in particular,
		\[
		\specd H_\alpha=\emptyset \text{ for any }\alpha\ge 0.
		\]
		
		As for any $\lambda<0$ the operator $S(\lambda)$ is compact, self-adjoint, non-negative, and injective (see the beginning of Section~\ref{sec2}), we can enumerate its eigenvalues $\mu_n\big(S(\lambda)\big)$ counting the multiplicities such that
		\[
		0<\mu_1\big(S(\lambda)\big)\le \mu_2\big(S(\lambda)\big)\le\dots
		\]
		and for each $n\in\NN$ one shows that the maps
		\begin{align*}
			(-\infty,0)\ni \lambda&\mapsto \mu_n\big(S(\lambda)\big)\in (0,\infty),\\	
			(-\infty,0)\ni \lambda&\mapsto \lambda \mu_n\big(S(\lambda)\big)\in (-\infty,0)
		\end{align*}
		are continuous and strictly monotonically increasing with
		\[
		\lim_{\lambda\to-\infty} \lambda \mu_n\big(S(\lambda)\big)=-\infty,
		\quad
		\lim_{\lambda\to 0^-} \lambda \mu_n\big(S(\lambda)\big)=0.
		\]
		In fact,  the argument of \cite[Proposition 2.2(i)]{bhs} for smooth $\Sigma$ still applies, as the proof of the above properties is
		only based on the integration by parts and some estimates from \cite{GS}, which are also valid for Lipschitz $\Sigma$. It follows that for any $n\in\NN$ and any $\alpha<0$ the equation
		\[
		\alpha \lambda \mu_n\big(S(\lambda)\big)=1
		\]
		has a unique solution $\lambda_n\big(H_\alpha)$, and by Theorem \ref{thm2} the number $\lambda_n(H_\alpha)$ is the $n$-th negative eigenvalue
		of $H_\alpha$, if enumerated in the non-increasing order with multiplicities taken into account. This shows that $\specd H_\alpha$ is an infinite discrete subset of $(-\infty,0)$ and does not accumulate at $0$. At the same time, $\specd H_\alpha$ cannot have any accumulation point in $(-\infty,0)$: any such accumulation point would be in the essential spectrum of $H_\alpha$ in contraction to \eqref{sess}. As an infinite sequence has at least one (infinite or finite) accumulation point, it follows that $\lim\limits_{n\to\infty}\lambda_n(H_\alpha)=-\infty$ for any $\alpha<0$.
	\end{remark}
	
	The discussion in this section shows that the qualitative spectral picture	for $H_\alpha$ with Lipschitz $\Sigma$ (essential spectrum, cardinality of the discrete spectrum) is the same as for smooth $\Sigma$. In order to show that the non-smoothness produced some new effects, we will look at the asymptotic behavior of $\lambda_n\big(H_\alpha\big)$ as $\alpha\to 0^-$ in the next section.

	\section{Relations between oblique transmission conditions and $\delta$-potentials}\label{sec3}
	
	As already mentioned in the introduction, in \cite{bhs} it was shown that for any fixed $n\in\NN$ one has
	\begin{equation}
		\label{lnha}
		\lambda_n(H_\alpha)=-\frac{4}{\alpha^2}+O(1) \text{ for $\alpha\to 0^-$, if $\Sigma$ is smooth.}
	\end{equation}
	The proof of this fact was based on a comparison of the eigenvalues of $H_\alpha$
	with the eigenvalues of Schr\"odinger operators with $\delta$-potentials supported by $\Sigma$. Namely, for $\beta\in\rr$
	and a simple Lipschitz curve $\Gamma\subset\RR^2$, either closed or with a controlled behavior at infinity (which holds in all subsequent examples), consider
	the sesquilinear form $q_\beta^{\Gamma}$ given by
	\begin{align}\label{schrodinger}
		q_\beta^\Gamma(f,f)= \int_{\rr^2} |\nabla f|^2\,\dd x +\beta \int_\Sigma |\gamma f|^2\, \dd\sigma, \quad \dom q_\beta^\Gamma= H^1(\rr^2),
	\end{align}
	which is closed and semibounded below and defines a unique self-adjoint operator $Q^\Gamma_\beta$ in $L^2(\RR^2)$, see \cite[Section 3.2]{bel}.
	We will be mainly interested in the case $\Gamma:=\Sigma$, and we abbreviate
	\[
	q_\beta:=q^\Sigma_\beta,
	\qquad
	Q_\beta:=Q^\Sigma_\beta,
	\]
	but some other auxiliary curves $\Gamma$ will be used at various intermediate steps of analysis below.
	
	As $\Sigma$ is compact, one has $\spece Q_\beta=[0,\infty)$ and $\specd Q_\beta$ is finite, and
	for any $\lambda<0$ one has the equivalence
	\begin{equation} \label{BirS}
		\dim\ker (Q_\beta-\lambda)=\dim\ker \big( 1 + \beta S(\lambda)\big),
	\end{equation}
	in particular,
	\[
	\lambda\in \specd Q_\beta \iff 0\in  \specd \big(1 + \beta S(\lambda)\big),
	\]
	see \cite[Lemma 2.3 and Theorem 4.2]{BEKS}. By taking the ordering into account
	one concludes that for any $\lambda<0$, $\beta<0$ and $n\in\NN$ one has the equivalence
	\[
	\lambda=\mu_n(Q_\beta) \iff -\dfrac{1}{\beta}=\mu_n\big(S(\lambda)\big).
	\]

	Therefore, the discrete spectra of both $H_\alpha$ and $Q_\beta$ are closely related
	to each other, as both can be determined with the help of the eigenvalues of the integral operator $S(\lambda)$. The derivation of \eqref{lnha} in \cite{bhs} was based on the analysis of $Q_\beta$ for smooth $\Sigma$ in \cite{EY}: For each fixed $n\in\NN$ the operator
	$Q_\beta$ has at least $n$ negative eigenvalues for all sufficiently large negative $\beta$, and the $n$th eigenvalue $\mu_n(Q_\beta)$,
	if enumerated in the non-decreasing order with multiplicities taken into account,  behaves as
	\begin{equation}
		\label{mnq}
		\mu_n(Q_\beta)=-\frac{1}{4}\beta^2+O(1) \text{ for $\beta\to -\infty$, if $\Sigma$ is $C^4$-smooth.}
	\end{equation}
	Surprisingly, no similar results seem known for non-smooth $\Sigma$: While some conjectures
	can be found in the review \cite{leaky1}, we are not aware of any sufficient progress. In this section, we will
	prove some basic estimates for the eigenvalues of $Q^\Sigma_\beta$ with Lipschitz or piecewise smooth $\Sigma$,
	which will in turn produce new results for the eigenvalues of $H_\alpha$.

	\begin{proposition}\label{propas1} Let $0<a<b$ and $n\in\NN$ be such that $Q_\beta$ has at least $n$ negative eigenvalues for all sufficiently large negative $\beta$ and
		\begin{equation}
			\label{as1}
			-b\beta^2\le \mu_n(Q_\beta)\le -a\beta^2 \text{ for }\beta\to-\infty,
		\end{equation}
		then
		\begin{equation}
			\label{as2}
			-\frac{b}{a^2}\dfrac{1}{\alpha^2}\le \lambda_n(H_\alpha)\le -\frac{a}{b^2}\dfrac{1}{\alpha^2} \text{ for }\alpha\to 0^-.
		\end{equation}
	\end{proposition}	
	
	\begin{proof}
		We have $\mu_n\big( S\big(\mu_n(Q_\beta)\big)\Big)=-\frac{1}{\beta}$, therefore,
		\[
		\mu_n(Q_\beta)\mu_n\big( S\big(\mu_n(Q_\beta)\big)\Big)=-\frac{\mu_n(Q_\beta)}{\beta},
		\]
		and in view of \eqref{as1} we have
		\begin{equation}
			\label{ccp}
			b\beta\le \mu_n(Q_\beta)\mu_n\big( S\big(\mu_n(Q_\beta)\big)\Big)\le a\beta.
		\end{equation}
		
		Furthermore (see Remark \ref{rmk9}),
		\begin{align*}
			a\beta&=\dfrac{1}{\frac{1}{a\beta}}=\lambda_n(H_{\frac{1}{a\beta}})\mu_n\Big( S\big(\lambda_n(H_{\frac{1}{a\beta}})\big)\Big),\\
			b\beta&=\dfrac{1}{\frac{1}{b\beta}}=\lambda_n(H_{\frac{1}{b\beta}})\mu_n\Big( S\big(\lambda_n(H_{\frac{1}{b\beta}})\big)\Big),
		\end{align*}
		and the substitution into \eqref{ccp} gives
		\begin{align*}
			\lambda_n(H_{\frac{1}{b\beta}})\mu_n\Big( S\big(\lambda_n(H_{\frac{1}{b\beta}})\big)\Big)&\le \mu_n(Q_\beta)\mu_n\big( S\big(\mu_n(Q_\beta)\big)\Big)\\
			&\le \lambda_n(H_{\frac{1}{a\beta}})\mu_n\Big( S\big(\lambda_n(H_{\frac{1}{a\beta}})\big)\Big).
		\end{align*}
		As the function $(-\infty,0)\ni\lambda\mapsto \lambda \mu_n\big(S(\lambda)\big)$ is increasing, it follows
		\begin{equation}
			\label{ccp2}
			\lambda_n(H_{\frac{1}{b\beta}})\le \mu_n(Q_\beta)\le \lambda_n(H_{\frac{1}{a\beta}}) \text{ as }\beta\to-\infty.
		\end{equation}
		The reparametrization $\beta:=\frac{1}{b\alpha}$ with $\alpha\to 0^-$ in the left-hand side of \eqref{ccp2} gives
		\[
		\lambda_n(H_{\alpha})\le \mu_n(Q_{\frac{1}{b\alpha}})\stackrel{\eqref{as1}}{\le} -\dfrac{a}{b^2}\,\dfrac{1}{\alpha^2}.
		\]
		Similarly, the reparametrization $\beta:=\frac{1}{a\alpha}$ with $\alpha\to 0^-$ in the right-hand side of \eqref{ccp2} gives
		\[
		\lambda_n(H_{\alpha})\ge \mu_n(Q_{\frac{1}{a\alpha}})\stackrel{\eqref{as1}}{\ge} -\dfrac{b}{a^2}\,\dfrac{1}{\alpha^2}.\qedhere
		\]	
	\end{proof}
	
	\begin{corollary}\label{corol01}
		If for some $b>0$ and $n\in\NN$ one has
		\[
		\mu_n(Q_\beta)=-b\beta^2+o(\beta^2) \text{ for }\beta\to-\infty,
		\]
		then
		\[
		\lambda_n(H_\alpha)=-\frac{1}{b\alpha^2}+o\big(\dfrac{1}{\alpha^2}\big) \text{ for }\alpha\to 0^-.
		\]
	\end{corollary}

	\begin{proof}
		For any $\varepsilon\in(0,1)$ one has
		\[
		-(1+\varepsilon)b\beta^2\le\mu_n(Q_\beta)=-(1-\varepsilon)b\beta^2 \text{ for }\beta\to-\infty,
		\]
		and	Proposition~\ref{propas1} shows that for $\alpha\to 0^-$ there holds
		\[
		-\dfrac{1+\varepsilon}{(1-\varepsilon)^2}\cdot\dfrac{1}{b\alpha^2}
		\le
		\lambda_n(H_\alpha)\le-\dfrac{1-\varepsilon}{(1+\varepsilon)^2}\cdot\dfrac{1}{ b\alpha^2}.
		\]
		As $\varepsilon$ can be chosen arbitrarily small, the claim follows.
	\end{proof}	
	
	\begin{remark}
		By combining the last corollary with \eqref{mnq} we see that for smooth $\Sigma$ one has
		$\lambda_n(H_\alpha)=4\alpha^{-2}+o(\alpha^{-2})$ for $\alpha\to0^-$, which is the result of \cite{bhs} with a weaker remainder estimate (which can be easily improved by a small refinement of the above constructions).
		
		We are now going to show that the dependence of the eigenvalues of $Q_\beta$ on $\beta$ can be different from \eqref{mnq}, which in turn gives a different asymptotic for $\lambda_n(H_\alpha)$.
	\end{remark}

	For the analysis let us recall the min-max principle. If $Q$ is lower semibounded  self-adjoint operator in an infinite-dimensional Hilbert space $\mathcal{H}$, generated by a closed sesquilinear form $q$, then for any $n\in\NN$ one defines
	\[
	\Lambda_n(Q):=\inf_{\substack{F\subset\dom q\\ \dim F=n}}\sup_{f\in F\setminus\{0\}}\dfrac{q(f,f)}{\|f\|^2}.
	\]
	It is known that $n\mapsto \Lambda_n(Q)$ is non-decreasing with
	\[
	\lim\limits_{n\to\infty}\Lambda_n(Q)=\inf\spece Q\quad \text{($=\infty$ for $\spece Q=\emptyset$)},
	\]
	and for any $n\in\NN$ with $\Lambda_n(Q)<\inf\spece Q$ the number $\Lambda_n(Q)$ is the $n$-th eigenvalue of $Q$ if counted in the non-decreasing order with multiplicities taken into account, see \cite[Section~XIII.1]{RS}.

	\begin{proposition}\label{exieig}
		For any $n\in\nn$ there are $\beta_n<0$ and $0<a<b$ such that for all $\beta<\beta_n$ the operator $Q_\beta$
		has at least $n$ negative eigenvalues (if counted with multiplicities), and
		\begin{equation}
			\label{aan}
			-b\beta^2\le \mu_n(Q_\beta)\le -a \beta^2 \text{ for all } \beta<\beta_n.
		\end{equation}
	\end{proposition}
	
	\begin{proof} The standard trace inequality, see e.g. \cite[Theorem~1.5.1.10]{Gris}, implies that there exists $c>0$ such that
		\begin{multline}\label{grsv1}
			\int_{\Sigma} |\gamma^+ f_+|^2\dd\sigma\le c\bigg(\varepsilon \int_{\Omega_+} |\nabla f_+|^2\dd x + \dfrac{1}{\varepsilon}\int_{\Omega_+} |f_+|^2\dd x\bigg)\\
			\text{ for all } f_+\in H^1(\Omega_+) \text{ and } \varepsilon\in(0,1).
		\end{multline}	
		For any $f\in H^1(\RR^2)$, $\beta<0$ and $\varepsilon\in(0,1)$ we obtain then
		\begin{align*}
			q_\beta(f,f)&=\int_{\Omega_+}|\nabla f_+|^2\, \dd x+\int_{\Omega_-}|\nabla f_-|^2\, \dd x-|\beta|\int_{\Sigma}|\gamma^+ f_+|^2\, \dd \sigma\\
			&\ge \int_{\Omega_+}|\nabla f_+|^2\, \dd x-|\beta| c\bigg(\varepsilon \int_{\Omega_+} |\nabla f_+|^2\dd x + \dfrac{1}{\varepsilon}\int_{\Omega_+} |f_+|^2\dd x\bigg)\\
			&=(1-|\beta| c \varepsilon)\int_{\Omega_+}|\nabla f_+|^2\, \dd x-\dfrac{|\beta| c}{\varepsilon}\int_{\Omega_+} |f_+|^2\dd x.\bigg.
		\end{align*}
		The above inequalities are valid, in particular, for any $\beta<-c^{-1}$ and $\varepsilon:=(-c\beta)^{-1}$. In that case $1+\beta c\varepsilon=0$, and
		\[
		q_\beta(f,f)\ge -(\beta c)^2\int_{\Omega_+} |f_+|^2\dd x\ge -c^2\beta^2 \|f\|^2_{L^2(\RR^2)},
		\]
		showing
		\begin{equation}
			\label{l1q}
			\Lambda_1(Q_\beta)\ge -c^2\beta^2 \text{ for all }\beta< -c^{-1}.
		\end{equation}

		Let $n\in\NN$. Pick $n$ distinct points $p_1,\dots,p_n\in \Sigma$.
		Let $j\in\{1,\dots,n\}$, then there exist constants $A,B>0$ and a Lipschitz function $h_j:(-A,A)\to (-\frac{B}{2},\frac{B}{2})$
		and an orthogonal coordinate change
		\[
		\Phi_j: (-A,A)\times(-B,B)\ni (y_1,y_2)\mapsto \RR^2,\quad \Phi_j(0,0)=p_j,
		\]
		such that
		\begin{multline*}
			\Phi_j\big((-A,A)\times(-B,B)\big)\cap \Sigma\\
			=\Phi_j\big(\big\{(y_1,y_2):\, y_1\in(-A,A),\, y_2=h_j(y_1)\big\}\big),
		\end{multline*}
		see \cite[Definition 1.2.1.1]{Gris}.
		By taking sufficiently small $A$ one may additionally assume that the $n$ rectangles $V_j:=\Phi_j\big((-A,A)\times (-B,B)\big)$, $j\in\{1,\dots,n\}$, are mutually disjoint and $A,B$ are independent of $j$, and denote
		$M:=\max_j \|h'_j\|_\infty$.
		Let $g\in C^\infty_c(\RR^2)$ with $\supp g \subset (-A,A)\times(-\frac{B}{2},\frac{B}{2})$,  $g\ge 0$ and $g(0,0)=1$.
		Let $\theta\in(0,1)$ (a precise value will be chosen later) and for all negative $\beta$ with $|\beta|>\theta^{-1}$ define $f_j:\RR^2\to \CC$
		such that $\supp f_j\subset V_j$ and
		\[
		f_j\big(\Phi(y_1,y_2)\big)=g\Big(y_1,\theta|\beta|\big(y_2-h_j(y_1)\big)\Big),\  (y_1,y_2)\in(-A,A)\times(-B,B).
		\]
		As $f_j$ is Lipschitz with compact support, one has $f_j\in H^1(\RR^2)$, and
		\begin{align*}
			\int_{\RR^2}|f_j|^2\,\dd x&=\int_{V_j}|f_j|^2\,\dd x=\int_{(-A,A)\times(-B,B)} \big|(f_j\circ \Phi_j)\big|^2\,\dd y\\
			&=\int_{-A}^{A} \int_\RR \Big|g\Big(y_1,\theta|\beta|\big(y_2-h_j(y_1)\big)\Big)\Big|^2\, dy_2\, \dd y_1\\
			&= (\theta|\beta|)^{-1}\int_{-A}^{A} \int_\RR \big|g(y_1,z)\big|^2\, \dd z\, \dd y_1=(\theta |\beta|)^{-1}\|g\|_{L^2(\RR^2)}^2,\\
			\int_\Sigma |f_j|^2\, \dd\sigma&=\int_{-A}^A \big|g(y_1,0)\Big|^2\sqrt{1+|h'_j(y_1)|^2}\,\dd y_1\ge \|g(\cdot,0)\big\|^2_{L^2(-A,A)}.
		\end{align*}	
		We further have
		\begin{align*}	
			\big|\nabla (f_j\circ \Phi_j(y)\big|^2&=\Big|\partial_1 g\Big(y_1,\theta|\beta|\big(y_2-h_j(y_1)\big)\Big)\\
			&\qquad-\theta|\beta| h'_j(y_1)\partial_2 g\Big(y_1,\theta|\beta|\big(y_2-h_j(y_1)\big)\Big)\Big|^2\\
			&\qquad +
			\Big|\theta|\beta| \partial_2 g\Big(y_1,\theta|\beta|\big(y_2-h_j(y_1)\big)\Big|^2\\
			&\le 4(M^2+1)(\theta\beta)^2 \Big|\nabla g\Big(y_1,\theta|\beta|\big(y_2-h_j(y_1)\big)\Big)\Big|^2,
		\end{align*}
		therefore,
		\begin{align*}
			\int_{\RR^2}&|\nabla f_j|^2\,\dd x=\int_{V_j}|\nabla f_j|^2\, \dd x=\int_{(-A,A)\times(-B,B)} \big|\nabla (f_j\circ \Phi_j)\big|^2\,\dd y\\
			&\le 4(M^2+1)(\theta\beta)^2\int_{(-A,A)\times(-B,B)}\Big|\nabla g\Big(y_1,\theta|\beta|\big(y_2-h_j(y_1)\big)\Big)\Big|^2\,\dd y\\
			&\le 4(M^2+1)(\theta\beta)^2 \int_{-A}^{A} \int_\RR \Big|\nabla g\Big(y_1,\theta|\beta|\big(y_2-h_j(y_1)\big)\Big)\Big|^2\, \dd y_2\, \dd y_1\\
			&\le 4(M^2+1)(\theta\beta)^2\cdot (\theta|\beta|)^{-1}\int_{-A}^{A} \int_\RR \big|\nabla g(y_1,z)\big|^2\, \dd z\, \dd y_1\\
			&= 4(M^2+1)\theta|\beta|\|\nabla g\|_{L^2(\RR^2)}^2.
		\end{align*}
		It follows
		\begin{align*}
			\dfrac{q_\beta(f_j,f_j)}{\|f_j\|^2_{L^2(\RR^2)}}&\le \dfrac{4(M^2+1)\theta|\beta|\|\nabla g\|_{L^2(\RR^2)}^2 - |\beta| \|g(\cdot,0)\big\|^2_{L^2(-A,A)}}{(\theta|\beta|)^{-1}\|g\|_{L^2(\RR^2)}^2}=c_\theta \beta^2,\\
			c_\theta&:=\dfrac{4(M^2+1)\theta^2\|\nabla g\|_{L^2(\RR^2)}^2-\theta\|g(\cdot,0)\big\|^2_{L^2(-A,A)}}{\|g\|_{L^2(\RR^2)}^2},
		\end{align*}
		and we can choose $\theta>0$ sufficiently small to obtain $c_\theta<0$. Remark that $c_\theta$ is independent of $j$.
		
		Let $F:=\mathop{\mathrm{span}}\{f_1,\dots,f_n\}$.
		As $f_1,\dots,f_n$ have mutually disjoint support, for any $f=b_1f_1+\dots+b_n f_n\in F$ with $b_j\in\CC$ we have
		\begin{align*}
			q_\beta(f,f)&=\sum_{j=1}^n |b_j|^2q_\beta(f_j,f_j)\le c_\theta\beta^2\sum_{j=1}^n |b_j|^2  \|f_j\|^2_{L^2(\RR^2)}
			=c_\theta\beta^2\|f\|^2_{L^2(\RR^2)},
		\end{align*}
		hence,
		\[
		\Lambda_n(Q_\beta)\le \sup_{f\in F\setminus\{0\}}\dfrac{q_\beta(f,f)}{\|f\|^2_{L^2(\RR^2)}}\le c_\theta \beta^2<0=\inf\spece Q_\beta.
		\]
		This implies that $\Lambda_j(Q_\beta)=\mu_n(Q_\beta)$ for all $j\in\{1,\dots,n\}$,
		and by combining with \eqref{l1q} we obtain \eqref{aan} with $a:=-c_\theta$ and $b:=c^2$.
	\end{proof}
	
	By combining Propositions \ref{propas1} and \ref{exieig} we arrive at the following two-sided estimate:
	
	\begin{corollary}
		For any $n\in\NN$ there are $0<A<B$ such that
		\[
		-\dfrac{B}{\alpha^2}\le 
		\lambda_n(H_\alpha)\le-\dfrac{A}{\alpha^2} \text{ for }\alpha\to 0^-.
		\]
	\end{corollary}
	
	It is not clear if the result can be improved. In particular, we do not know if the assumption of Corollary \ref{corol01} holds at least for some $n$ for general Lipschitz $\Sigma$. Nevertheless, we are going to show that at least for some $\Sigma$ and $n=1$ the asymptotic \eqref{mnq} fails.
	
	\section{Analysis of a piecewise smooth curve with a corner}\label{sec4}	
	
	Let us pick $\theta\in(0,\frac{\pi}{2})$ and denote 
	\begin{align*}
		\Gamma_\theta:=\big\{(r\cos\omega,r\sin\omega)\in\rr^2:\, r\ge 0,\  |\omega|=\theta \big\},
	\end{align*}
	which is the union of two half-lines meeting at the origin with the angle $2\theta$ between them.
	For the associated operator $Q^\Gamma_\beta$ with $\beta<0$ it is known that
	\[
	\spece Q^{\Gamma_\theta}_\beta=[-\frac{\beta^2}{4},\infty),
	\quad
	\specd Q^{\Gamma_\theta}_\beta\ne \emptyset,
	\]
	in particular,
	\[
	-b_\theta:=\mu_1(Q^{\Gamma_\theta}_{-1})<-\frac{1}{4},\qquad
	\mu_1(Q^\Gamma_\beta)=-b_\theta\beta^2 \text{ for all }\beta<0, 
	\]
	see \cite{EY,KP}.
	
	\begin{theorem}\label{thm13}
		Let $\Sigma\subset\RR^2$ be a simple closed Lipschitz curve such that
		\begin{equation}
			\label{assump}
			\begin{minipage}{0.85\textwidth}
				\begin{itemize}
					\item[(i)] for some $r>0$ and $\theta\in(0,\frac{\pi}{2})$ there holds
					\[
					\Sigma\cap B_{2r}(0)=\Gamma_\theta\cap B_{2r}(0),
					\]
					\item[(ii)] $\Sigma$ is $C^4$ smooth at all points except at the origin,
				\end{itemize}
			\end{minipage}
		\end{equation}
		then
		\[
		\mu_1(Q^\Sigma_\beta)=-b_\theta\beta^2+O(1) \text{ for }\beta\to-\infty.
		\]	
	\end{theorem}
	
	\begin{proof}
		During the proof denote $\Gamma:=\Gamma_\theta$ and $b:=b_\theta$.
		Let $\chi_1,\chi_2\in C^\infty(\rr^2)$ be such that $0\leq \chi_j\leq 1$, $\chi_1^2+\chi_2^2=1$, with 
		\begin{align*}
			\chi_1(x)=1 \text{ for } |x|<r\quad \text{and}\quad \chi_2(x)=1 \text{ for } |x|>2r,
		\end{align*}
		and denote $C:=\| |\nabla \chi_1|^2 +|\nabla \chi_2|^2\|_{\infty}<\infty$. Then a direct computation shows that
		for any $f\in H^1(\RR^2)$ there holds
		\begin{align}\label{eq1}
			\begin{split}
				q_{\beta}^{\Sigma}(f,f)&=q_{\beta}^{\Sigma}(\chi_1 f,\chi_1 f) + q_{\beta}^{\Sigma}(\chi_2 f,\chi_2 f)\\
				&\qquad -\int_{\rr^2} ( |\nabla \chi_1|^2 +|\nabla \chi_2|^2) |f|^2 \ddd x\\
				&\geq q_{\beta}^{\Sigma}(\chi_1 f,\chi_1 f) + q_{\beta}^{\Sigma}(\chi_2 f,\chi_2 f) - C\| f\|^2_{L^2(\rr^2)}.
			\end{split}
		\end{align}
		
		Due to the assumption (ii) in \eqref{assump} by rounding the corner at the origin we can construct a simple closed $C^4$ smooth curve $\Sigma_0$ such that $\Sigma\cap \big(\rr^2 \setminus B_r(0)\big)= \Sigma_0\cap \big(\rr^2 \setminus B_r(0)\big)$, then obviously
		\begin{align}
			q_{\beta}^{\Sigma}(\chi_2 f,\chi_2 f) = q_{\beta}^{\Sigma_0}(\chi_2 f,\chi_2 f) \text{ for all } f\in H^1(\rr^2).
		\end{align}
		
		Furthermore, let $\Tilde{Q}_{\beta}^{\Sigma}$ be the self-adjoint operator in $L^2\big(B_{2r}(0)\big)$
		generated by the sesquilinear form $\tilde{q}_{\beta}^{\Sigma}$ given by
		\[
		\Tilde{q}_{\beta}^{\Sigma}(g,g) = \int_{B_{2r}(0)} |\nabla g|^2\ddd x + \beta \int_{B_{2r}(0)\cap\Sigma} |g|^2\dd\sigma 	
		\]
		on the domain $\dom \Tilde{q}_{\beta}^{\Sigma}=H^1_0\big(B_{2r}(0)\big)$. As each function in $\dom \Tilde{q}_{\beta}^{\Sigma}$
		can be extended by zero to a function in $\dom q^\Sigma_\beta$, we have by the min-max principle
		\begin{equation}
			\label{ll-01}
			\Lambda_1(Q^\Sigma_\beta)\le \Lambda_1(\Tilde{Q}_{\beta}^{\Sigma}),\quad \beta\in\RR.
		\end{equation}		
		For any $f\in H^1(\RR^2)$ one has $\chi_1 f\in H^1_0\big(B_{2r}(0)\big)$
		with
		\[
		q_{\beta}^{\Sigma}(\chi_1 f,\chi_1 f) = \Tilde{q}_{\beta}^{\Sigma}(\chi_1 f,\chi_2f).
		\]
		By combining these observations with \eqref{eq1} we obtain
		\[
		q_{\beta}^{\Sigma}(f,f)+ C\| f \|^2_{L^2(\rr^2)} \geq  \tilde{q}_{\beta}^{\Sigma}(\chi_1 f,\chi_1 f) + q_{\beta}^{\Sigma_0}(\chi_2 f,\chi_2 f),\ f\in H^1(\RR^2).
		\]
		The right-hand side is the sesquilinear form of the operator $\Tilde{Q}_{\beta}^{\Sigma}\oplus Q_\beta^{\Sigma_0}$
		computed on $\big((\chi_1 f,\chi_2 f),(\chi_1 f,\chi_2 f)\big)$ and that
		the map $f\mapsto (\chi_1 f,\chi_2 f)$ preserves the $L^2$-norm, so the min-max principle implies
		\[
		\Lambda_1(Q^\Sigma_\beta)-C\ge \Lambda_1(\Tilde{Q}_{\beta}^{\Sigma}\oplus Q_\beta^{\Sigma_0}),\quad \beta\in\RR.
		\]
		and we have $\Lambda_1(\Tilde{Q}_{\beta}^{\Sigma}\oplus Q_\beta^{\Sigma_0})=\min\Big\{\Lambda_1(\Tilde{Q}_{\beta}^{\Sigma}),\Lambda_1(Q_\beta^{\Sigma_0})\Big\}$,
		so in combination with \eqref{ll-01}
		we arrive at
		\begin{equation}
			\label{eq3}
			\min\Big\{\Lambda_1(\Tilde{Q}_{\beta}^{\Sigma}),\Lambda_1(Q_\beta^{\Sigma_0})\Big\}-C\le
			\Lambda_1(Q_{\beta}^{\Sigma})\le \Lambda_1(\Tilde{Q}^\Sigma_\beta),\quad \beta\in\RR.
		\end{equation}
		
		The same argument can be applied to $Q^\Gamma_\beta$ instead of $Q^\Sigma_\beta$,
		which gives
		\begin{equation}
			\label{eq4}
			\min\Big\{\Lambda_1(\Tilde{Q}_{\beta}^{\Gamma}),\Lambda_1(Q_\beta^{\Gamma_0})\Big\}-C\le
			\Lambda_1(Q_{\beta}^{\Gamma})\le \Lambda_1(\Tilde{Q}^\Gamma_\beta),\quad \beta\in\RR,
		\end{equation}
		where $\Gamma_0$ is a $C^4$-smooth curve coinciding with $\Gamma$ in $\RR^2\setminus B_{r}(0)$
		and $\Tilde{Q}_{\beta}^{\Gamma}$ is the self-adjoint operator in $L^2\big(B_{2r}(0)\big)$ generated by 
		by the sesquilinear form $\Tilde{q}_{\beta}^{\Gamma}$ given by
		\[
		\Tilde{Q}_{\beta}^{\Gamma}(g,g) = \int_{B_{2r}(0)} |\nabla g|^2\ddd x + \beta \int_{B_{2r}(0)\cap\Gamma} |g|^2\dd\sigma
		\]
		on the domain $\dom \Tilde{q}_{\beta}^{\Gamma}=H^1_0\big(B_{2r}(0)\big)$. 
		We have noted above for any $\beta<0$ we have $\Lambda_1(Q_{\beta}^{\Gamma})=-b\beta^2$ (while $b>\frac{1}{4}$),
		and by \cite[Section 5]{EY2} for $\beta\to-\infty$ one has
		\[
		\Lambda_1(Q_\beta^{\Gamma_0})=\mu_1(Q_\beta^{\Gamma_0})=-\frac{1}{4}\beta^2+O(1).
		\]
		In particular, for $\beta\to-\infty$ one has $\Lambda_1(Q_\beta^{\Gamma_0})-C> \Lambda_1(Q_{\beta}^{\Gamma})$,
		and \eqref{eq4} implies
		\[
		\Lambda_1(\Tilde{Q}_{\beta}^{\Gamma})-C\le \Lambda_1(Q_{\beta}^{\Gamma})\le \Lambda_1(\Tilde{Q}^\Gamma_\beta),\quad \beta\to-\infty,
		\]
		in particular,
		\[
		\Lambda_1(\Tilde{Q}_{\beta}^{\Gamma})=\Lambda_1(Q_{\beta}^{\Gamma})+O(1)\equiv -b\beta^2+O(1) \text{ for }\beta\to -\infty.
		\]
		We further remark that due to the assumption (i) in \eqref{assump} we have $\Tilde{Q}_{\beta}^{\Gamma}=\Tilde{Q}_{\beta}^{\Sigma}$,
		and by \eqref{mnq} for $\Sigma:=\Sigma_0$ for $\beta\to-\infty$ we have
		\[
		\Lambda_1(Q^{\Sigma_0}_\beta)=\mu_1(Q^{\Sigma_0}_\beta)=-\frac{1}{4}\beta^2+O(1)>\Lambda_1(\Tilde{Q}_{\beta}^{\Sigma}),
		\]
		so \eqref{eq3} yields that for $\beta\to-\infty$ one has
		\[
		\Lambda_1(\Tilde{Q}_{\beta}^{\Sigma})-C\le
		\Lambda_1(Q_{\beta}^{\Sigma})\le \Lambda_1(\Tilde{Q}^\Sigma_\beta),
		\]
		in particular,
		\[
		\Lambda_1(Q_{\beta}^{\Sigma})=\Lambda_1(\Tilde{Q}^\Sigma_\beta)+O(1)\equiv \Lambda_1(\Tilde{Q}^\Gamma_\beta)+O(1)=-b\beta^2+O(1).
		\]
		The right-hand side is negative (i.e. lies below the bottom of the essential spectrum of $Q^\Sigma_\beta$), so
		$\Lambda_1(Q_{\beta}^{\Sigma})=\mu_1(Q_{\beta}^{\Sigma})$ for all sufficiently large negative $\beta$.
	\end{proof}
	
	By combining Theorem \ref{thm13} with Corollary \ref{corol01} we arrive at the following observation:
	
	\begin{corollary}\label{corol02}
		If $\Sigma$ satisfies \eqref{assump}, then for the respective operator $H_\alpha$ one has
		\[
		\lambda_1(H_\alpha)=-\dfrac{1}{b_\theta \alpha^2} +o\Big( \dfrac{1}{\alpha^2}\Big),\quad \alpha\to 0^-.
		\]	
	\end{corollary}

	\begin{theorem}
		For any $B\in(1,4)$ there exists a simple closed Lipschitz curve $\Sigma$ such that for the associated operator $H_\alpha$
		there holds
		\[
		\lambda_1(H_\alpha)=-\dfrac{B}{\alpha^2}+o\Big(\dfrac{1}{\alpha^2}\Big) \text{ for } \alpha\to 0^-.
		\]
	\end{theorem}
	
	\begin{proof}
		By \cite[Section~5]{en1} the function $(0,\frac{\pi}{2})\ni\theta\mapsto (-b_\theta)$ is continuous and monotonically decreasing, and
		by \cite[Theorem 1.11]{dr} and \cite[Theorem 2.2]{ek} one has
		\[
		\lim_{\theta\to 0^+} (-b_\theta)=-1,\quad
		\lim_{\theta\to \frac{\pi}{2}^-} (-b_\theta)=-\frac{1}{4}.
		\]
		Therefore, for any $B\in(1,4)$ there exists $\theta\in(0,\frac{\pi}{2})$ with $b_\theta=\frac{1}{B}$,
		and the claim follows by Corollary \ref{corol02}.	
	\end{proof}
	
	\section*{Acknowledgments}Badreddine Benhellal and Konstantin Pankrashkin were supported by the Deutsche For\-schungsgemeinschaft (German Research Foundation, DFG) -- 491606144.
	Miguel Camarasa was supported by the Ministry of Science and Innovation of Spain (BCAM Severo Ochoa accreditation CEX2021-001142-S/MICIN/ AEI/10.13039/ 501100011033) and by the Basque Government through the BERC 2022--2025 program. 
	
	This work was prepared during a visit of Miguel Camarasa to the Carl von Ossietzky Universit\"at Oldenburg in January--April 2024, and he thanks the Institut f\"ur Mathematik for the warm hospitality.

\end{document}